\documentclass[11pt]{amsart}
\usepackage{fullpage}
\usepackage[active]{srcltx}
\usepackage{amssymb,amsthm,amsmath,amscd}

\usepackage{hyperref}

\usepackage{graphicx} 
\usepackage{amsfonts}

\usepackage{tikz-cd}
\usepackage[all,cmtip]{xy}

\usepackage{amsmath}
\usepackage{amssymb}
\usepackage{amsthm}
\newtheorem{theorem}{Theorem}
\newtheorem{lemma}[theorem]{Lemma}

\newtheorem*{theoremM}{Main Theorem}

\theoremstyle{definition}
\newtheorem{remark}[theorem]{Remark}

\DeclareMathOperator{\m}{\mathfrak{m}}

\DeclareMathOperator{\depth}{\text{depth}}
\DeclareMathOperator{\cd}{\text{cd}}
\DeclareMathOperator{\pd}{\text{pd}}
\DeclareMathOperator{\Ext}{\text{Ext}}
\DeclareMathOperator{\Hom}{\text{Hom}}
\DeclareMathOperator{\Image}{\text{Image}}

\newcommand{\cdim}{\operatorname{cd}}
\newcommand{\height}{\operatorname{ht}}

\title{Local cohomological dimension and depth in mixed characteristic}
\author{Linquan Ma}
\address{Department of Mathematics, Purdue University, West Lafayette, IN 47907, USA}
\email{ma326@purdue.edu}

\begin{document}

\begin{abstract}
Let $(R,\mathfrak m)$ be an unramified regular local ring of mixed characteristic $(0,p)$ and dimension $d$ and let $I\subseteq R$ be an ideal. We prove that $\depth(R/I)\geq 3$ implies $\cd(I)\leq d-3$, and if $R$ is essentially of finite type over a DVR, then $\depth(R/I)\geq 4$ implies $\cd(I)\leq d-4$. More generally, $H_I^j(R)$ is a $\mathbb Q$-vector space whenever $j>d-\depth(R/I)$, thus vanishing of local cohomology in this range is determined completely by the characteristic zero fiber.
\end{abstract}

\maketitle

\section{Introduction}

Local cohomology was introduced by Grothendieck in the early 1960s as a tool for studying the local structure of schemes and the cohomology of
their complements.  Let $R$ be a Noetherian  ring and
let $I\subseteq R$ be an ideal.  The local cohomological dimension of $I$ is defined as
\[
    \cdim(I)=\sup\{j\mid H_I^j(R)\neq 0\}.
\]
In other words, $\cdim(I)$ records the last possible
nonvanishing degree of local cohomology supported on $V(I)$. 


A problem in commutative algebra and algebraic geometry, originally posed by Grothendieck in \cite{HartshorneLocalCohomology}, is to find algebraic or geometric conditions ensuring the vanishing of $H_I^j(R)$. In this direction, Grothendieck proved the first general bound:
\[
  \height(I)\leq \cdim(I)\leq \dim(R).
\]

The next improvement is the Hartshorne--Lichtenbaum vanishing theorem
\cite{HartshorneCohomologicalDimension}. Namely, if $R$ is a Noetherian complete local domain, then
\[
    \cd(I)\leq \dim(R)-1
    \,\ \Longleftrightarrow \,\
    \dim(R/I)\geq 1.
\]
For an arbitrary Noetherian local ring one can obtain a statement by considering all minimal primes of its completion, see \cite[{Tag 0EB6}]{stacks-project}. In sum, the Hartshorne--Lichtenbaum vanishing essentially characterizes $\cd(I)\leq \dim(R)-1$.

Beyond this point, the study of local cohomological dimension (and the vanishing of local cohomology) has focused primarily on the case where $R$ is nonsingular. Suppose $(R,\m)$ is a complete
regular local ring with separably closed residue field. Then the second vanishing theorem says that
\[
 \cdim(I)\leq \dim(R)-2
 \,\ \Longleftrightarrow \,\
 \dim(R/I)\geq 2 \text{ and } 
 \operatorname{Spec}(R/I)\setminus\{\mathfrak m\} \text{ is connected}.
\]
For an arbitrary regular local ring $R$, this criterion can be applied after
passing to the completion of the strict Henselization of $R$ and thus leads to a characterization of $\cd(I)\leq \dim(R)-2$ for nonsingular $R$. Hartshorne proved the theorem in the graded case and called it ``second
vanishing theorem'' in \cite{HartshorneCohomologicalDimension}. The equal characteristic zero
case was proved by Ogus \cite{OgusLocalCohomologicalDimension} and the equal characteristic $p>0$ case was proved
by Peskine--Szpiro \cite{PeskineSzpiroIHES}. Huneke--Lyubeznik later gave a uniform treatment in equal characteristic
\cite{HunekeLyubeznikVanishingLocalCohomology}. The unramified mixed characteristic case was established by Zhang \cite{ZhangSecondVanishing}, and the remaining ramified mixed characteristic case was recently settled by 
Scheffelin \cite{ScheffelinSecondVanishingTheorem}.

These results suggest a natural hierarchy of $\cd(I)$ through the depth of the quotient ring $R/I$. More specifically, let $R$ be a regular local ring of dimension $d$. If $\depth(R/I)\geq 1$ then $\dim(R/I)\geq 1$ and thus
Hartshorne--Lichtenbaum vanishing implies $\cdim(I)\leq d-1$.  If
$\depth(R/I)\geq 2$, then Hartshorne's connectedness theorem
\cite[Theorem 2.2]{HartshorneCompleteIntersectionConnectedness} together with the second vanishing theorem implies
$\cdim(I)\leq d-2$. For regular local rings of characteristic $p>0$, this pattern continues: Peskine--Szpiro \cite[Ch.~III, \S4]{PeskineSzpiroIHES} proved the following inequality relating local cohomological dimension and depth in this case:
\[
    \cdim(I)\leq d -\depth(R/I).
\]

On the other hand, this inequality does not hold in equal characteristic zero. In fact, it already fails for generic determinantal ideals \cite[Example~2.6]{VarbaroCohomologicalDimension}: if $R$ is a power series ring over $\mathbb{Q}$ with a $2\times 3$ matrix of variables and $I\subseteq R$ is the ideal generated by the $2\times 2$ minors, then $\depth(R/I)=4$ while $\cd(I)=3 > 6-4 =d-\depth(R/I)$. Nevertheless, the
depth-three case does hold at least under mild finiteness hypotheses: it was proved by Dao--Takagi \cite{DaoTakagiDepthCohomologicalDimension} (extending earlier work of Varbaro \cite{VarbaroCohomologicalDimension} in the graded case) that if $R$ is a regular local ring essentially of finite type over a field of characteristic zero and $I\subseteq R$, then $\depth(R/I)\geq 3$ implies $\cd(I)\leq d-3$.

In this short note, we establish such
depth-three vanishing in the unramified mixed characteristic case, and in fact, we can also prove a depth-four vanishing under mild finiteness hypotheses (so the scenario in mixed characteristic is better than that in characteristic zero).

\begin{theoremM}
Let $R$ be an unramified regular local ring of mixed characteristic and of dimension $d$, and let $I\subseteq R$ be an ideal. 
\begin{enumerate}
    \item If $\depth(R/I)\geq 3$, then $\cd(I)\leq d-3$.
    \item If $\depth(R/I)\geq 4$ and $R$ is essentially of finite type over a DVR, then $\cd(I) \leq d-4$.
\end{enumerate}
\end{theoremM}

This follows from our Theorem~\ref{thm: depth and cohomological dimension}, which says a bit more: namely that beyond $\depth(R/I)$, the obstruction to the vanishing of $H_I^j(R)$ lies entirely in characteristic zero. The Main Theorem then follows from this statement together with the existing vanishing theorems mentioned above. It is easy to see that our result is sharp in the sense that this pattern stops at $\depth(R/I)=5$ by the same generic determinantal example (over $\mathbb{Z}_p$), see Remark~\ref{rmk: sharpness}.

\subsection*{Acknowledgment} 
The author was supported by NSF grant DMS-2302430. I would like to thank Bhargav Bhatt, Anurag Singh,
and Uli Walther for many enjoyable discussions on local cohomology. I also thank Manav Batavia and Hanlin Cai for some discussions related to the results of the paper.

\subsection*{AI disclosure} The main idea of the proof was inspired by a conversation between the author and OpenAI's ChatGPT 5.6 Pro. Specifically, the author originally tried to construct counterexamples to the depth-three vanishing in mixed characteristic. ChatGPT then pointed out that a candidate example the author proposed indeed satisfies the vanishing and gave an essentially valid argument by analyzing the maps from Ext to local cohomology and utilizing $D$-modules. The author then realized that the argument could be extended and combined with techniques in \cite{BBLSZ1} to prove the general case. The final writing was done by the author (ChatGPT was used to correct grammer and other minor mistakes).

\section{The main result}

Recall that a regular local ring $(R,\m)$ of mixed characteristic $(0,p)$ is unramified if $p\notin \m^2$, or equivalently, $R/p$ is a regular local ring of characteristic $p$. We refer the readers to \cite{LyubeznikFmodules} and \cite{BBLSZ1} for basic facts on $F$-modules and $D$-modules. The following is the key ingredient in the proof of our main result.  

\begin{lemma}
\label{lem: key surjectivity}
Let $(R,\m)$ be an unramified regular local ring of mixed characteristic $(0,p)$ and dimension $d$. Suppose $I\subseteq R$ is an ideal so that $\depth(R/I)=g$. Then we have 
\begin{enumerate}
    \item[(i)] $H_I^{j}(R/p)=0$ for all $j>d-g$.
    \item[(ii)] The natural map $H_I^{d-g}(R) \to H_I^{d-g}(R/p)$ is surjective. 
\end{enumerate}
\end{lemma}
\begin{proof}
By passing to a faithfully flat extension of the same dimension (which does not affect the hypotheses and the vanishing/surjectivity of local cohomology), we may assume that $(R,\m)$ is complete with perfect residue field, i.e., $R\cong W(k)[[x_2,\dots,x_d]]$ where $k$ is a perfect field. By the Auslander--Buchsbaum formula, we know that $\pd_R(R/I)=d-g$. Let 
$$0\to R^{n_{d-g}}\xrightarrow{\phi} R^{n_{d-g-1}}\to \cdots \to R^{n_1}\to R\to 0$$
be a minimal free resolution of $R/I$. 
Note that for any $R$-module $M$, we have
$$\Ext^{d-g}_R(R/I, M) \cong \text{coker}(M^{\oplus n_{d-g-1}}\xrightarrow{\phi^{\vee}}M^{\oplus n_{d-g}}).$$
In particular, we have 
$$\Ext^{d-g}_R(R/I, R/p)\cong \Ext^{d-g}_R(R/I, R)/p\cdot \Ext^{d-g}_R(R/I, R).$$ 
Therefore we have
\begin{enumerate}
\item[(a)] $\Ext^j_R(R/I,R/p)=0$ for all $j>d-g$.
\item[(b)] The natural map $\Ext^{d-g}_R(R/I, R) \to \Ext^{d-g}_R(R/I, R/p)$ is surjective.
\end{enumerate} 
We next note that for all $j$, there are natural maps
\begin{equation*}
\tag{$\dagger$}
\Ext_{R/p}^j(R/(I+p), R/p)\to \Ext_R^j(R/I, R/p)\to H_I^j(R/p) 
\end{equation*}
Here the first map is obtained by taking $j$-th cohomology of the canonical map
$$R\Hom_{R/p}(R/(I+p), R/p) \to R\Hom_{R/p}((R/I)\otimes^{\mathbb{L}}_RR/p, R/p) \cong R\Hom_R(R/I, R/p)$$
and it is easy to see that this is compatible with the map to local cohomology. 
It follows from ($\dagger$) and (a) that if $j>d-g$, then the natural map 
$$\Ext_{R/p}^j(R/(I+p), R/p)\xrightarrow{\gamma} H_I^j(R/p)$$
is zero. Since $R/p$ is a regular local ring of characteristic $p$ and $\Image(\gamma)$ generates $H_I^j(R/p)$ as an $F$-module by \cite[Proposition 1.11]{LyubeznikFmodules}, it follows that $H_I^j(R/p)=0$ for $j>d-g$. This completes the proof of (i).

To prove (ii), we consider the commutative diagram:
\[\xymatrix{
 & \Ext^{d-g}_{R/p}(R/(I+p), R/p) \ar[d] \ar@/^5pc/[dd]^{\beta}\\
\Ext^{d-g}_R(R/I, R) \ar[d] \ar@{->>}[r] & \Ext^{d-g}_R(R/I, R/p) \ar[d]^{\alpha} \\
H_I^{d-g}(R) \ar[r]^{\varphi} & H_I^{d-g}(R/p).
}
\]
where the first horizontal surjection follows from (b) and the existence of right vertical maps follows from ($\dagger$). Chasing this diagram we find that
$$\Image(\varphi) \supseteq \Image(\alpha) \supseteq \Image(\beta).$$
Finally, we note that by \cite[Equation (2.1) and Proof of Theorem 3.1]{BBLSZ1}, $\Image(\varphi)$ is a $D(R/p)$-submodule of $H_I^{d-g}(R/p)$. Since $R/p$ is a regular local ring of characteristic $p$, by \cite[Proposition 1.11]{LyubeznikFmodules}, $\Image(\beta)$ generates $H_I^{d-g}(R/p)$ as an $F$-module and hence also as a $D(R/p)$-module by \cite[Corollary 4.4]{ABLGeneratorsDmodules} (note that we have reduced to a situation that $R/p$ is $F$-finite). Putting these together, we know that $\Image(\varphi)=H_I^{d-g}(R/p)$, i.e., $\varphi$ is surjective. This completes the proof of (ii).
\end{proof}

The next lemma is well known (for example, see \cite[Lemma 2.7]{JiangTestElements}), and we include a short proof for completeness.

\begin{lemma}
\label{lem: height of maximal ideal}
Let $(R,\m)$ be an equidimensional and catenary Noetherian local ring of dimension $d$. Suppose $z\in \m$ is not in any minimal prime of $R$. Then every maximal ideal of $R[1/z]$ has height $d-1$. 
\end{lemma}
\begin{proof}
Let $Q$ be a maximal ideal of $R[1/z]$. We will view $Q$ as a prime ideal of $R$ that does not contain $z$. If $Q$ has height $h$, then by prime avoidance we can choose $x_1,\dots,x_h,z$ that is part of a system of parameters of $R$ and $Q$ is a minimal prime of $(x_1,\dots,x_h)$. If $h\neq d-1$, then we can extend $x_1,\dots,x_h,z$ to a full system of parameters $x_1,\dots,x_h,z, x_{h+1},\dots,x_{d-1}$ so that $z,x_{h+1},\dots,x_{d-1}$ is a system of parameters on $R/Q$ (here we use the equidimensional and catenary assumptions to guarantee that $\dim(R/Q) = d-h$). But then any minimal prime of $Q+(x_{h+1},\dots,x_{d-1})$ is a prime ideal that contains $Q$ but does not contain $z$, contradicting the maximality of $Q$.
\end{proof}

We now state and prove our main result.

\begin{theorem}
\label{thm: depth and cohomological dimension}
Let $R$ be an unramified regular local ring of mixed characteristic $(0,p)$ and dimension $d$. Suppose $I\subseteq R$ is an ideal so that $\depth(R/I)=g$. Then $H_I^j(R)$ is a $\mathbb{Q}$-vector space for every $j> d-g$. Moreover,
\begin{enumerate}
    \item[(i)] If $g\geq 3$, then $\cd(I)\leq d-3$.
    \item[(ii)] If $g\geq 4$ and $R$ is essentially of finite type over a DVR, then $\cd(I) \leq d-4$.
\end{enumerate}
\end{theorem}
\begin{proof}
Consider the long exact sequence of local cohomology:
$$\cdots \to H_I^{j-1}(R)\to H_I^{j-1}(R/p) \to H_I^j(R) \xrightarrow{\cdot p} H_I^j(R) \to H_I^j(R/p) \to \cdots .$$
By Lemma~\ref{lem: key surjectivity} we know that multiplication by $p$ is bijective for all $j>d-g$. This proves the first statement.

We next prove the ``Moreover" part of the theorem. Let $Q$ be a maximal ideal of $R[1/p]$. By Lemma~\ref{lem: height of maximal ideal}, we have $\dim(R_Q)=\depth(R_Q)=d-1$. By the Auslander--Buchsbaum formula, we know that $\pd_R(R/I)=d-g$ and thus $\pd_{R_Q}(R_Q/IR_Q)\leq d-g$. It follows that $$\depth(R_Q/IR_Q)\geq (d-1)-(d-g) =g-1.$$

Now in case (i), by the first statement already established, it suffices to show that $H_I^j(R[1/p])=0$ for $j\geq d-2$. If $Q$ is any maximal ideal of $R[1/p]$, then as $\dim(R_Q)=d-1$ and $\depth(R_Q/IR_Q)\geq g-1\geq 2$, we know that 
$H_I^{d-1}(R_Q)=H_I^{d-2}(R_Q)=0$
by the Hartshorne--Lichtenbaum vanishing, Hartshorne's connectedness theorem and the second vanishing theorem. 

Finally in case (ii), as above, it suffices to show that $H_I^j(R[1/p])=0$ for $j\geq d-3$. Now if $Q$ is any maximal ideal of $R[1/p]$, then $R_Q$ is a regular local ring of dimension $d-1$ essentially of finite type over a field of characteristic zero and $\depth(R_Q/IR_Q)\geq g-1\geq 3$. Thus $H_I^{d-3}(R_Q)=0$ by \cite[Corollary 2.8]{DaoTakagiDepthCohomologicalDimension}, while $H_I^{d-1}(R_Q)=H_I^{d-2}(R_Q)=0$ follows from the Hartshorne--Lichtenbaum vanishing,  Hartshorne's connectedness theorem and the second vanishing theorem. 
\end{proof}

\begin{remark}
\label{rmk: sharpness}
With notations as in Theorem~\ref{thm: depth and cohomological dimension}. It is not true that $g\geq 5$ implies $\cd(I)\leq d-5$. For example, let $R=\mathbb{Z}_p[x_1, x_2, x_3, y_1, y_2, y_3]_{\m}$ where $\m$ denotes the maximal ideal $(p,x_i, y_i | 1\leq i\leq 3)$ and let $I=(x_iy_j - x_jy_i \mid  1\leq i < j \leq 3)$. Then we have $d=\dim(R)=7$ and $g=\depth(R/I) =5$, but 
$H_I^3(R) \neq 0$. In fact, by \cite[Theorem 1.2]{LyubeznikSinghWalther} (together with a harmless completion), we have 
$$H_I^3(R) \cong  H_I^3(R)[1/p] \cong H_{(x_i, y_i | 1\leq i\leq 3)}^6(\mathbb{Q}_p[[x_1, x_2, x_3, y_1, y_2, y_3]]).$$
\end{remark}

\begin{remark}
With notations as in Theorem~\ref{thm: depth and cohomological dimension}. If $R/I$ is Cohen--Macaulay and $p$ is a nonzerodivisor on $R/I$, then multiplication by $p$ is injective on $H_I^j(R)$ for every $j$. Indeed, as $R/(I+p)$ is Cohen--Macaulay, $H_I^{d-g}(R/p)$ is the only nonvanishing local cohomology of $R/p$ supported at $I$ by \cite{PeskineSzpiroIHES} and the claim follows by examining the long exact sequence and using Lemma~\ref{lem: key surjectivity} as in Theorem~\ref{thm: depth and cohomological dimension}. This recovers and generalizes some aspects of \cite[Theorem 2.1]{JPSWArithmeticRank}.   
\end{remark}

\begin{remark}
When $(R,\m)$ is a regular local ring of characteristic $p>0$, Hartshorne--Speiser \cite{HartshorneSpeiserLocalCohomologicalDimension} and Lyubeznik \cite{LyubeznikVanishingLocalCohomology} obtained formulas for the local cohomological dimension. For example, by \cite[Theorem 4.3]{LyubeznikVanishingLocalCohomology}, we have 
$$\cd(I) = \dim(R)- \depth(R/I)_{\text{perf}}.$$ Here $\depth(R/I)_{\text{perf}} = \min\{j \mid H_{\m}^j((R/I)_{\text{perf}}) \neq 0\}$ where $(R/I)_{\text{perf}} = \varinjlim_e F^e_*(R/I)$. This formula implies Peskine--Szpiro's inequality $\cd(I)\leq \dim(R)-\depth(R/I)$ because $\depth(R/I)_{\text{perf}}\geq \depth(R/I)$. It was shown in \cite[Theorem 6.6]{ZhangSecondVanishing} that if $R$ is an unramified regular local ring of mixed characteristic $(0,p)$ and $p\in I$, then 
$$\cd(I) \leq \dim(R)- \depth(R/I)_{\text{perf}}.$$
We point out that in general, this inequality can be strict. For example, let $(R,\m)=\mathbb{Z}_2[[x_1,\dots,x_6]]$ and let $J$ be the square-free monomial ideal corresponding to the standard triangulation of $\mathbb{R}\mathbb{P}^2$, see \cite[Example 5.2]{SinghWaltherBockstein}, and let $I=J+(2)$. Then as $R/I$ is $F$-pure, we have 
$$\depth(R/I)_{\text{perf}}=\depth(R/I)=2$$
where the second equality follows from Hochster's formula and a computation on singular cohomology (see \cite[Example 5.10]{SinghWaltherBockstein}). On the other hand, by \cite[Remark 4.3 and Proposition 4.5]{DattaSwitalaZhang} we know that $H_J^{>4}(R)=0$ and $H_J^4(R)\neq 0$ and is annihilated by $(2)$. It follows from the long exact sequence $\cdots\to H_I^i(R) \to H_J^i(R)\to H_J^i(R[1/2])\to\cdots$ that $H_I^{>4}(R)=0$ and $H_I^4(R)\neq 0$. Thus
$$\cd(I)=4 < 7-2 = \dim(R)-\depth(R/I)_{\text{perf}}.$$
This example also answers \cite[Question 6.2]{ZhangSecondVanishing} in the negative: $H_I^5(R)=0$ but the Frobenius is not nilpotent on $H_{\m}^2(R/I)$.
\end{remark}

Our result leaves open a natural question: if $R$ is a {ramified} regular local ring of mixed characteristic and $I\subseteq R$ is an ideal so that $\depth(R/I)\geq 3$, then is it true that $\cd(I)\leq \dim(R)-3$? Our current techniques do not seem able to tackle this case. 

\bibliographystyle{alpha}
\bibliography{Bib}
\end{document}